\documentclass[reqno,11pt]{amsart}
\usepackage{amsmath}
\usepackage{amssymb}
\usepackage{latexsym}
\usepackage{amsthm}
\usepackage{fullpage}
\usepackage[colorinlistoftodos]{todonotes}
\usepackage{marginnote}
\usepackage{enumitem}
\usepackage{mathtools}
\usepackage{dsfont}
\usepackage{mdframed}
\usepackage{xcolor}
\usepackage{hyperref}
\usepackage{esint}
\usepackage{amsfonts,latexsym,amsmath,amscd,geometry}
\numberwithin{equation}{section}

\newcommand{\R}{\mathbb{R}}

\newcommand{\eps}{\varepsilon}

\newtheorem{theorem}{Theorem}[section]

\newtheorem{remark}[theorem]{Remark}

\numberwithin{equation}{section}
\numberwithin{figure}{section}

\newcommand{\abs}[1]{\left\vert{#1}\right\vert}

\def\intave#1{\int_{#1}\hbox{\llap{$\raise2.3pt\hbox{\vrule
height.9pt width7pt}\phantom{\scriptstyle{#1}}\mkern-2mu$}}}

\begin{document}
\author{Fang-Hua Lin}
\address{Courant Institute of Mathematical Sciences, New York University, NY 10012, USA}
\email{linf@math.nyu.edu}
\author{Changyou Wang}
\address{Department of Mathematics, Purdue University, West Lafayette, IN 47907, USA}
\email{wang2482@purdue.edu}
\date{February 26,  2026}
\title{Remarks on the heat flow of harmonic maps into CAT(0)-spaces}
\begin{abstract} 
In this paper, we present an alternate, elementary proof of the local Lipschitz regularity of the suitable weak solution of heat flow of harmonic maps into CAT(0)-metric spaces, whose existence was established by Lin, Segatti, Sire, and Wang through an elliptic regularization approach. The ideas of the proof are inspired by Korevaar and Schoen,
and they work for any CAT(0)-metric space $(X,d)$ as the target and any complete Riemanan manifold $(M,g)$, with positive injectivity radius and bounded curvature, as the domain. 
\end{abstract}
\maketitle

\section{Introduction}

In the seminal paper \cite{Eells-Sampson1964}, Eells and Sampson first introduced the heat flow of harmonic maps between compact Riemannian manifolds $(M, g)$ and $(N, h)$ without boundaries. When the sectional curvature of $(N, h)$ is non-positive, it was shown by \cite{Eells-Sampson1964} that for any initial map $u_0\in C^\infty(M,N)$, there exists a unique smooth solution $u\in C^\infty(M\times [0,\infty), N)$ of the heat flow of harmonic map such that $u(\cdot, t)\xrightarrow{t_k\nearrow\infty}u_\infty$ in $C^2(M, N)$ and
$u_\infty$ is a minimizing harmonic map in the free homotopy class
$[u_0]$ (see also related work by Hartman \cite{Hartman1967} and Hamilton \cite{Hamilton1975}). 

In his 1984's paper \cite{Schoen1984}, Schoen established, under the assumption $(N,h)$ has non-positive curvature,
the convexity of Dirichlet energy functional $E(u(t)): [0,1]\to [0,\infty)$ for any $u\in C^\infty(M\times [0,1],N)$ satisfying that
for $x\in M$, $u(x,\cdot):[0,1]\to N$ is a constant speed geodesic. This in turn provides a variational approach of the existence of minimizing harmonic maps in any free homotopy class $\alpha\in [M, N]$,
when $(N,h)$ has non-positive curvature. In order to establish rigidity theorems of geometric groups,  Gromov-Schoen \cite{GromovSchoen1992} initiated the study of harmonic map into  
non-positively curved (NPC) metric spaces $(X,d)$, and showed, in particular, its Lipschitz regularity when $(X,d)$ are Euclidean buildings 
(i.e., locally finite Riemannian simplicial complexes). Subsequently, Korevaar-Schoen \cite{Korevaar-Schoen1993, Korevaar-Schoen1997} developed 
a general theory of minimizing harmonic maps into CAT(0)-metric spaces and established their Lipschitz regularity, see also the work
by Jost \cite{Jost1994}. Since then, we have seen many interesting development on harmonic maps into CAT(0)-metric spaces, especially further allowing the domains to be Alexandrov spaces with curvature bounded from below. We refer the interested reader to Chen \cite{Chen1995},  Jost \cite{Jost1996}, Lin \cite{Lin1997}, Daskalopoulos-Messe \cite{DM2009, DM2010, DM2012}, Zhang-Zhu \cite{ZhangZhu2018}, Gigli \cite{Gigli2022},
and Mondino-Semola \cite{Mondino-Semola2022}.

On the other hand, it has been a natural, but difficult problem how to extend Eells-Sampson's theorem and develop a successful theory of heat flow of harmonic maps into CAT(0)-metric spaces. There were fewer
results available in the literature. U. Mayer \cite{Mayer1998} extended the Crandall-Liggett method to construct global weak solutions
of the negative gradient flow of Dirichlet energy functional for maps into CAT(0)-metric spaces, that enjoy the semi-group property, but left
the question of regularity for such solutions unanswered. Very recently, in a joint work with A. Segatti and Y. Sire \cite{LSSW2025}, we developed a new method, namely, the elliptic regularization approach of the heat flow of harmonic maps, and proved the existence of a unique global suitable weak solution of the heat flow of harmonic maps into CAT(0)-metric spaces $(X,d)$, that is H\"older continuous when $(X,d)$ additionally is locally compact. Furthermore, we introduced the parabolic version of Almgren type frequency functions for such a heat flow of harmonic maps and showed it is a monotonically increasing function. As an immediate application of such a monotonicity property, we proved the spatial Lipschitz regularity of any suitable weak solution of the heat flow of harmonic maps into a wide class of CAT(0)-metric spaces. 

Now we briefly review the main results by \cite{LSSW2025}. First, through this paper, $(X,d)$ is assumed to be a CAT(0)-metric space, that is, $(X,d)$ is a complete and separable metric space, that satisfies
the following two conditions:
\begin{itemize}
\item[(i)] $(X,d)$ is a complete length space: For any pair of points $P, Q\in X$, $d(P,Q)$ equals to the length of a rectifiable curve $\gamma:[0,1]\to X$
such that $\gamma(0)=P$ and $\gamma(1)=Q$. Such a $\gamma$ is called a  geodesic between $P$ and $Q$.

\item[(ii)] For any three points $P, Q, R\in X$, let $\gamma_{Q,R}$ be a geodesic between $Q$ and $R$. For $0<\lambda<1$,
let $Q_\lambda\in\Gamma_{Q,R}$ be such that $d(Q_\lambda, Q)=\lambda d(Q,R)$ and $d(Q_\lambda, R)=(1-\lambda)d(Q,R)$, then it holds that 
\begin{equation}\label{NPC1}
d^2(P, Q_\lambda)\le (1-\lambda)d^2(P,Q)+\lambda d^2(Q, R)-\lambda(1-\lambda)d^2(Q,R).
\end{equation}
\end{itemize}

Let $(M,g)$ be a compact (or complete) $n$-dimensional Riemannian manifold without boundary. We closely follow Korevaar-Schoen \cite{Korevaar-Schoen1993} 
for the definitions of Sobolev spaces between $M$ (or $M\times (0,\infty)$) and $X$, and of the Dirichlet energy density 
$|\nabla u|^2$ (or $|\partial_t u|^2$) as well as the directional energy density $|u_*(Z)|^2$ for a map $u\in H^1(M,X)$, where $Z$ is a smooth tangential vector field $Z$ on $M$. Recall the (generalized) Dirichlet energy functional $E: L^2(M, X)\to [0,\infty]$ is defined by
\[
E(u)=\begin{cases}\frac12 \int_M |\nabla u|^2\,dv_g, & u\in H^1(M,X),\\
+\infty &  \mbox{ otherwise in } L^2(M.X),
\end{cases}
\]
with domain $D(E)=H^1(M,X)$.
For $u_0\in H^1(M,X)$ and $\eps>0$, let $u_\eps\in \mathfrak{V}_{u_0}$ be the (unique) minimizer of the Weighted Energy Dissipation (or WED) energy functional 
\begin{equation}
\label{eq:wed_intro_0}
\mathcal{I}_\eps(v)=\frac{1}{2}\iint_{M\times\R_+}\frac{e^{-t/\eps}}{\eps}\left(\eps\abs{\partial_t v}^2+\abs{\nabla v}^2\right)\,dv_gdt
\end{equation}
over the configuration function space
\begin{equation}
\label{eq:Hut0}
{\mathfrak{V}_{{u}_0}= \left\{v:M\times\R_+\to X: \,v(0)= {u}_0,  
\,\,\iint_{M\times\R_+}(\abs{\partial_t v}^2 +\abs{\nabla v}^2)e^{-t/\eps}\,dv_gdt<+\infty\right\}}.
\end{equation}

It was shown by \cite[Section 4.3]{LSSW2025} that $\{u_\eps\}$ is  bounded in $H^1_{\rm{loc}}(M\times (0,\infty))$, that is
\begin{align}\label{energybounded}
\iint_{M\times (0,\infty)} |\partial_t u_\eps|^2\,dv_gdt\le E(u_0),
\ \ \ \iint_{M\times [0,T]} |\nabla u_\eps|^2\,dv_gdt\le (T+\eps)E(u_0),
\ \ \forall T>0.
\end{align}

Recall the induced $L^2$-distance on $L^2(M,X)$ by $(X,d)$ is given by
\[d_2(u,v)=\Big(\int_M d^2(u(x), v(x))\,dv_g\Big)^\frac12, 
\ \ \forall u,v\in L^2(M,X).\]

Recall that by \cite[Definition 1.2]{LSSW2025}, $u\in {\rm{AC}}^2\big([0,\infty), L^2(M,X))\cap L^\infty(0,\infty), H^1_{\rm{loc}}(M, X)\big)$ is called a suitable weak solution of the heat flow of harmonic maps from $(M,g)$ to $(X,d)$ with an initial condition
$u_0\in H^1(M,X)$, if $u\big|_{t=0}=u_0$ and if $u$ satisfies
the Evolution Variational Inequality (EVI):
\begin{align}\label{EVI0}
\frac12\frac{d}{dt} d_2^2(u(t), v)+E(u(t)\le E(v)
\ \ {\rm{in}}\ \ \mathcal{D}'(0,\infty), \ \forall v\in H^1(M,X).
\end{align}

The following theorem was proven by \cite{LSSW2025}.
\begin{theorem}\label{convergence} Let $(X,d)$ be a CAT(0)-metric space, 
and $u_0\in H^1(M, X)$. If $u_\eps={\rm{argmin}}\big\{\mathcal{I}_\eps(u),
\ u\in \mathfrak{V}_{u_0}\big\}$, then there exists a unique suitable weak solution $u$ of the heat flow of harmonic map from $(M,g)$ to $(X,d)$
such that
\[
u_\eps\xrightarrow{\eps\to 0} u \ \ {\rm{in}}\ \ L^2(M\times (0,\infty)).
\]
If $(X,d)$ additionally is locally compact, then $u\in C^\alpha(M\times (0,\infty))$ for some $0<\alpha<1$.   
\end{theorem}

Recall that there exists a nonnegative Radon measure $\nu=\nu_t dt$ on
$M\times (0,\infty)$ such that
\[
e_\eps(u_\eps)\,dxdt
=(\eps|\partial_t u_\eps|^2+|\nabla u_\eps|^2)\,dxdt
\rightharpoonup \mu=\mu_t\,dt:=(|\nabla u|^2+\nu_t)\,dt
\]
as weak convergence of Radon measures. 

When $(M,g)=(\R^n, dx^2)$, we introduced in \cite{LSSW2025} a parabolic frequency function for the suitable weak solution $u$ given by Theorem \ref{convergence}, that is, for $z_0=(x_0, t_0)\in \R^n\times (0,\infty)$
and $0<R\le \sqrt{t_0}$, 
\[
N(u; z_0, R, u(z_0))=\frac{E(u; z_0, R)}{H(u;z_0, R, u(z_0)},
\]
where
\[
E(u;z_0, R)=2R^2\int_{t=t_0-R^2} G_{z_0}(x,t)\,d\mu_t(x),
\]
and
\[
H(u; z_0, R, u(z_0))=
\int_{t=t_0-R^2}G_{z_0}(x,t)d^2(u(x,t), u(z_0))\,dx,
\]
where $G_{z_0}$ denotes the backward heat kernel in $\R^n\times (0,t_0)$.
From both the inner and outer variations of $\mathcal{I}_\eps$ at $u_\eps$, we proved in \cite[Therorem 4.4]{LSSW2025} that $N(u;z_0, R, u(z_0))$ is a monotonically increasing function of $R$
in $(0, \sqrt{t_0})$. In 
particular, $N(u;z_0, u(z_0))$ exists and is upper semi-continuous
in $\R^n\times (0,\infty)$. As an application of this monotonicity 
property of
frequency function, we were able to show the spatial Lipschitz regularity of $u$ for a wide class of CAT(0)-spaces including Euclidean buildings
and homogeneous trees. Namely, 
\begin{theorem}\label{lip-reg} 
If $u_0: \R^n\to (X,d)$, a CAT(0)-metric space,  satisfies
$E(u_0)<\infty$ and $d(u_0, Q)\in L^\infty(\R^n)$ for some $Q\in X$, 
then there exists a unique suitable weak solution  $u:\R^n\times (0,\infty)\to X$ of the heat flow of harmonic map, 
with initial data $u_0$, which is Lipschitz continuous in $x$ and $\frac12$-H\"older continuous in $t$.
\end{theorem}

From \cite{LSSW2025}, the local spatial Lipschitz continuity is a consequence of the fact that
the value of frequency function $N(u; z_0, u(z_0))\ge 1$ for all
$z_0\in \R^n\times (0,\infty)$.

Very recently, Zhang and Zhu \cite{ZhangZhu2026} proved the Lipschitz regularity of suitable weak solutions of heat flow of harmonic maps through a different method. In particular, 
they showed the local Lipschitz regularity for suitable weak solutions holds in $(x,t)$ when $(M,g)$ is a complete Riemannian manifold with 
Ricci curvature bounded from below and $(X,d)$ is a CAT(0)-space.
Roughly speaking, among other ingredients their arguments are 
mainly based on 1) the observation that $|\partial_tu|^2$ is a sub solution of the heat equation on
$M\times (0,\infty)$ so that for $t>0$, $\phi^t(x,y)=d(u(x,t), u(y,t))$
can be made into a sub-solution of the coupling elliptic equation 
$\Delta_{x,y}\phi^t(x,y)\ge -C$ on $M\times M$, and 2) the approximation of $\phi^t$ by the sup-convolution from the viscosity theory of nonlinear elliptic equations, that converts $\phi^t$ into a semi-convex sub-solution $\phi^t_\eps$ which has almost everywhere twice differentiability by Aleksandrov's theorem, this led them to establish a crucial parabolic perturbation Lemma. The idea of sup-convolution  
might date back to R. Jensen's original proof of uniqueness of viscosity solutions of second order fully nonlinear elliptic equations, see Jensen \cite{Jensen1988} or Caffarelli-Cabr\'e \cite{CCV43}. 

In this paper, we make a new observation of suitable weak solutions $u$ to the heat flow of harmonic maps. It is motivated by Korevaar-Schoen's 
derivation of weak Bochner type inequality for minimizing harmonic maps
into CAT(0)-spaces, which is based on Reshetnyak's quadrilateral comparison property for $d^2$, see \cite{Korevaar-Schoen1993}. More precisely, we show in Section 2 below that if $u$ satisfies the EVI \eqref{EVI0}, then
\begin{align}\label{EVI00}
&\iint_{M\times (0,\infty)} |\nabla u|^2(\partial_t \eta+2\Delta \eta+C\eta
+C|\nabla \eta|)\,dv_gdt\ge -C\iint_{M\times (0,\infty)}\eta |\partial_t u|^2\,dv_gdt
\end{align}    
holds for any $0\le \eta\in C_0^\infty(M\times (0,\infty))$,
with $C>0$ depending on the $C^{1,1}$-norm of $g$.

On the other hand, as already observed by \cite{ZhangZhu2026}, we have
that $|\partial_t u|^2$ satisfies 
\begin{align}\label{subcaloricity0}
(\partial_t-2\Delta) |\partial_t u|^2\ge 0 \ \ {\rm{in}}\ \ \mathcal{D}'(M\times (0,\infty)).   
\end{align}
See Section 3 below for a derivation slightly different from
\cite[Theorem 3.4]{ZhangZhu2026}. It is clear that \eqref{subcaloricity0} leads to the local $L^\infty$-bound of $|\partial_t u|^2$, which enables us to apply Moser's proof of Harnack inequality
(cf.  \cite{Moser1964})  to \eqref{EVI00}
to obtain the local $L^\infty$-bound of $|\nabla u|^2$.
This enables us to prove 
\begin{theorem}\label{Lip1}
Assume $(X,d)$ is a CAT(0)-metric space and $(M,g)$ is a complete Riemannian manifold with positive inj$(M)$ and bounded curvature.
For $u_0\in H^1(M, X)$, let $u: M\times [0,\infty)\to X$, with $u\in L^\infty_t H^1(M,X)$ and $\partial_t u\in L^2([0, T], L^2(M))$ for all $T<\infty$, be 
any suitable weak solution of the heat flow of harmonic maps from $(M,g)$ to $(X,d)$ with initial value $u_0$. Then $u$ is locally Lipschitz on $M\times (0,\infty)$. Furthermore,  for any
$\delta>0$, 
\begin{align}\label{global_lip}
|\partial_t u|^2+|\nabla u|^2\le C, \ \ {\rm{in}}\ \ M\times [\delta,\infty),
\end{align} 
where the constant $C>0$ depends on $\delta, \ E(u_0)$, inj$(M)$, and bound of the curvature of $g$.
\end{theorem}

As a byproduct of derivation of \eqref{EVI00}, we obtain the following 
apriori gradient estimate for any minimizer of WED functional $\mathcal{I}_\eps$ when the domain $(M,g)=(\R^n, dx^2)$ is an Euclidean space, which gives a positive answer to the question we asked how to extend \cite[Theorem 1.3]{LSSW2025} to CAT(0)-metric spaces. 
More precisely, we prove

\begin{theorem}[Uniform Lipschitz Estimate] \label{Theorem 1.4}
Let $(M, g)=(\R^n, dx^2)$  and $(X,d)$ be a CAT(0)-metric space. Let
$u_0\in H^1(\R^n, X)$ and $u_\eps\in \mathfrak{V}_{u_0}$  be a minimizer
of $\mathcal{I}_\eps$ over $\mathfrak{V}_{u_0}$. 
Then there exists $c=c(n)>0$ such that for any $z_0=(x_0, t_0)\in \R^n\times (0,\infty)$ and $0<r<\frac{\sqrt{t_0}}2\big\}$,
the following 
\begin{align}\label{x-lip}
\sup_{B_r(x_0)\times (t_0-r^2, t_0+r^2)}\big|\nabla u_\eps\big |^2  
\le c\big[\frac{\eps}{r^{n+2}}+\frac{1}{r^n}\big] E(u_0)
\end{align}
hold for all $0<\eps\le r^2$.
\end{theorem}

\section{Parabolic inequality of $|\nabla u|^2$}
Assume $(X,d)$ is CAT(0). To better demonstrate the ideas, 
we first consider the domain manifold $(M,g)$ is $(\R^n, g_0=dx^2)$,
and then present extra ingredients that are needed for the domain being
any general complete Riemann manifold without boundary.
\subsection{$(M,g)$: Euclidean space with standard metric}
Write $u=u_\eps$. For a small constant vector $e=(e,0)\in\R^{n+1}\setminus\{0\}$, define
\[u^e(x,t)=u(x+e,t), \ \ (x,t)\in \R^{n+1}_+.
\]
Denote $u_0=u$ and $u_1=u^e$.
Let $\eta\in C^\infty_0(\R^{n}\times (0,+\infty))$ be such that $0\le \eta
<\frac12$. For $(x,t)\in \R^{n+1}_+$, 
let $u_\eta(x,t)=u_{\eta(x,t)}(x,t)\in X$ be the unique point on the unique geodesic that joins $u_0(x,t)$ to $u_1(x,t)$ such that 
$$d(u_{\eta(x,t)}(x,t), u_0(x,t))=\eta(x,t)d(u_0(x,t), u_1(x,t)).$$
Then, since $\eta\big|_{t=0}=0$, $u_{\eta}\in \mathfrak{V}_{u_0}$.  Since $u$ minimizes $\mathcal{I}_\epsilon$ in $\mathfrak{V}_{u_0}$, we must have
$$
\mathcal{I}_\epsilon(u) \le \mathcal{I}_\epsilon(u_\eta).
$$

It is not hard to see from this definition
that for $(x,t)\in\R^{n+1}_+$, $u_{1-\eta(x,t)}(x,t)=u^e_{\eta(x,t)}(x,t)$. Since $u^e$ minimizes $\mathcal{I}_\epsilon$ over
$\mathfrak{V}_{u_0^e}$ and $u^e_\eta\in \mathfrak{V}_{u_0^e}$, 
it follows that 
$$
\mathcal{I}_\epsilon(u^e) \le \mathcal{I}_\epsilon(u^e_\eta)
=\mathcal{I}_\epsilon(u_{1-\eta}).
$$
Adding these two inequalities together implies
\begin{equation}\label{minimality}
\mathcal{I}_\epsilon(u)+\mathcal{I}_\epsilon(u^e)
\le \mathcal{I}_\epsilon(u_\eta)
+\mathcal{I}_\epsilon(u_{1-\eta}).
\end{equation}

Recall from \cite[Lemma 2.4.2]{Korevaar-Schoen1993} that
\begin{equation}\label{comparison}
\pi_{u_\eta}+\pi_{u_{1-\eta}}\le\pi_{u}+
\pi_{u^e}-\nabla\eta\otimes \nabla d^2(u, u^e)+Q(\eta,\nabla\eta),
\end{equation}
where $Q(\eta,\nabla\eta)$ consists of integrable terms that are quadratic in
$\eta$ and $\nabla\eta$, and the bilinear, nonnegative tensor $\pi_u:\Gamma(T(M\times(0,\infty)))\times \Gamma(T(M\times (0,\infty)))\to L^1(M\times (0,\infty))$ is defined by
\[
\pi_u(Z, W)=\frac14\big|u_*(Z+W)\big|^2-\frac14\big|u_*(Z-W)\big|^2.
\]
In particular, it follows from \eqref{comparison} that
\begin{align}\label{comparison-t}
|\partial_tu_\eta|^2+|\partial_tu_{1-\eta}|^2
\le |\partial_t u|^2+|\partial_t u^e|^2
-\partial_t \eta\partial_t(d^2(u, u^e))+Q(\eta,\partial_t\eta),
\end{align}
and
\begin{align}\label{comparison-x}
|\frac{\partial u_\eta}{\partial x_i}|^2+|\frac{\partial u_{1-\eta}}{\partial x_i}|^2
\le |\frac{\partial u}{\partial x_i}|^2+|\frac{\partial u^e}{\partial x_i}|^2
-\frac{\partial \eta}{\partial x_i}
\frac{\partial (d^2(u, u^e))}{\partial x_i}
+Q(\eta,\frac{\partial\eta}{\partial x_i}),\ i=1,\cdots, n.
\end{align}

Applying \eqref{comparison-t} and \eqref{comparison-x}
into the expansion of \eqref{minimality}, we obtain
\begin{align*}
&\iint_{\R^{n+1}_+} \frac{e^{-t/\eps}}{\eps}\left(\eps\abs{\partial_t u}^2+\abs{\nabla u}^2\right)\,dxdt 
+\iint_{\R^{n+1}_+} \frac{e^{-t/\eps}}{\eps}\left(\eps\abs{\partial_t u^e}^2+\abs{\nabla u^e}^2\right)\,dxdt\\
&\le 
\iint_{\R^{n+1}_+} \frac{e^{-t/\eps}}{\eps}\left(\eps\abs{\partial_t u_\eta}^2+\abs{\nabla u_\eta}^2\right)\,dxdt 
+\iint_{\R^{n+1}_+} \frac{e^{-t/\eps}}{\eps}\left(\eps\abs{\partial_t u_{1-\eta}}^2+\abs{\nabla u_{1-\eta}}^2\right)\,dxdt\\
&\le \iint_{\R^{n+1}_+} \frac{e^{-t/\eps}}{\eps}\left(\eps |\partial_t u|^2
+\eps|\partial_t u^e|^2-\eps \partial_t d^2(u, u^e)\cdot\partial_t \eta+\eps Q(\eta,\partial_t\eta)\right)\,dxdt\\
&+\iint_{\R^{n+1}_+} \frac{e^{-t/\eps}}{\eps}\left(|\nabla u|^2
+|\nabla u^e|^2-\nabla d^2(u, u^e)\cdot\nabla\eta+Q(\eta,\nabla\eta)\right)
\,dxdt.
\end{align*}
This yields
\begin{align*}
0\le  \iint_{\R^{n+1}_+}\frac{e^{-t/\eps}}{\eps}\Big[-\eps \partial_t d^2(u, u^e)\cdot\partial_t \eta+\eps Q(\eta,\partial_t\eta)-\nabla d^2(u, u^e)\cdot\nabla\eta+Q(\eta,\nabla\eta)\Big]
\,dxdt   
\end{align*}
Replacing $\eta$ by $\tau\eta$ in this inequality
and then sending $\tau\to 0$, we can eliminate the terms involving
the quadratic $Q$'s and obtain that 
\begin{align*}
0\le  \iint {e^{-t/\eps}}\Big[-\eps \partial_t d^2(u, u^e)\cdot\partial_t \eta-\nabla d^2(u, u^e)\cdot\nabla\eta\Big]
\,dxdt   
\end{align*}
Let $\eta=e^{\frac{t}{\eps}}\phi$. Then we arrive at
\begin{align}\label{subharmonic1}
0\le  \iint \Big[-\eps \partial_t d^2(u, u^e)\cdot\partial_t \phi
-\partial_t d^2(u, u^e)\phi-\nabla d^2(u, u^e)\cdot\nabla\phi\Big]
\,dxdt.  
\end{align}
This is equivalent to saying that
$0\le w=d^2(u, u^e)\in W^{1,2}_{\rm{loc}}(\R^{n+1}_+)$ is a weak solution of 
\begin{align}    
\mathcal{L}_\eps w:=\Big(\eps\partial^2_t-\partial_t +\Delta\big)w\ge 0.
\end{align}

For any small $h>0$ and the standard base vectors $e_i=\frac{\partial}{\partial x_i}$
for $i=1,\cdots, n$, let $e=he_i$
and consider the finite difference quotient 
$$(w_\eps)_i^h(x,t)=\frac{d^2(u_\eps(x+he_i,t), u_\eps(x,t))}{h^2},\ 
\ \forall (x,t)\in \R^n\times (0,\infty),$$
we have
\begin{align} \label{subharmonic2}   
\mathcal{L}_\eps \Big(\sum_{i=1}^n(w_\eps)_i^h\Big)
=\Big(\eps\partial^2_t-\partial_t +\Delta\Big)\Big(\sum_{i=1}^n(w_\eps)_i^h\Big)
\ge 0.
\end{align}
Applying the definition of energy density by
Korevaar-Schoen \cite{Korevaar-Schoen1993}, we can send $h\to 0$ in  
\eqref{subharmonic2} to obtain
\begin{align} \label{subharmonic3}   
\Big(\eps\partial^2_t-\partial_t +\Delta\Big)|\nabla u_\eps|^2\ge 0.
\end{align}

\bigskip
\noindent{\bf Proof of Theorem \ref{Theorem 1.4}}: With
\eqref{subharmonic3} at hand,  the proof follows exactly
from \cite[Lemma 6.3]{LSSW2025}. 
\qed

\begin{remark} As already proved by \cite{LSSW2025}
that there exists a unique suitable weak solution $u$ to the heat flow of harmonic maps such that
$$u_\eps\rightarrow u \ \ {\rm{in }}\ \ L^2_{\rm{loc}}\cap C^0_{\rm{loc}}(\R^{n+1}_+).$$
Sending $\eps\to 0$ \eqref{subharmonic3} 
and applying the lower semi-continuity property, 
we conclude that 
$|\nabla u|^2$ 
is a sub-caloric function, that is, 
\begin{align} \label{subharmonic6}   
\Big(-\partial_t+\Delta\Big)|\nabla u|^2\ge 0
\end{align}
holds in the weak sense.
\end{remark}

\subsection{$(M,g)$: complete Riemann manifold without boundary}

The main difficulty for the domain being a Riemannian manifold
arises from the fact that the translation of
a minimizer $u_\eps$ of $\mathcal{I}_\varepsilon$, $u_\eps^*(x,t)=u_\eps(\overline{x}(x,s),t)$, is a minimizer of $\mathcal{I}_\varepsilon$ associated with the pull-back metric $g^*
=\overline{x}^*(g)$ rather than $g$. In particular,
the errors of terms involving $|\partial_t u_\eps|^2$ are very difficult to
control. However,  as observed by Korevaar-Schoen \cite{Korevaar-Schoen1993} the errors of terms involving $|\nabla u_\eps|^2$ can be controlled as long as the metric $g\in C^{1,1}(M)$. This latter observation leads us to consider the Evolution Variational Inequality (or EVI) that is the very definition of $u$ being a suitable weak solution given by
\cite{LSSW2025}. 

First, recall that since any suitable solution $u$ to the heat flow of harmonic maps into $(X,d)$  satisfies the EVI \eqref{EVI0}, by integration with $t$ we have
\begin{align}\label{EVI1}
&\int_M d^2(u(x, t), w(x))\,dv_g-\int_M d^2(u(x, t-\delta), w(x))\,dv_g\nonumber\\
&\le 2\int_{t-\delta}^{t} (E(w;g)-E(u(\tau); g))\,d\tau\nonumber\\
&\le 2\delta (E(w;g)-E(u(t);g)), \ \forall w\in H^1(M, X), \ 0<\delta<t.
\end{align}
Here $E(u; g)$ is defined by 
$$E(u; g)=\int_M \frac12 g^{ij}\pi_u\big(\frac{\partial}{\partial x_i},
\frac{\partial}{\partial x_j}\big)\,dv_g,$$
and we have used the fact that $E(u(t); g)$ is monotonically decreasing
in the last inequality.

We follow closely the presentation of Korevaar-Schoen \cite[pages 633-638]{Korevaar-Schoen1993} along with suitable modifications. Let $\omega$ be a unit vector field on a local coordinate chart $\Omega$ of $M$. Let $\{\overline{x}(x,s)\}_{s\ge 0}$ be the flow generated by $\omega\in\Gamma(T\Omega)$, that is, 
\[
\overline{x}(x,0)=x;  \ \frac{d}{ds}\overline{x}(x,s)=\omega(\overline{x}(x,s)).
\]
For $s>0$ small, define 
$$u_{s\omega}(x,t)=u(\overline{x}(x,s), t),
\ \forall z=(x,t)\in \Omega\times (0,\infty).$$
Then it is not hard to see that $u_{s\omega}$ is a suitable weak solution of the heat flow of harmonic maps from
$(M, g_{s\omega}=\overline{x}(\cdot, s)^*g)$ to $(X,d)$. In particular, 
$u_{s\omega}$ satisfies \eqref{EVI1} 
with $g$ replaced by $g_{s\omega}$, that is,
\begin{align}\label{EVI2}
&\int_M d^2(u_{s\omega}(x, t), w(x))\,dv_{g_{s\omega}}-\int_M d^2(u_{s\omega}(x, t-\delta), w(x))\,dv_{g_{s\omega}}\nonumber\\
&\le 2\int_{t-\delta}^{t} (E(w;g_{s\omega})-E(u(\tau);g_{s\omega}))\,d\tau\nonumber\\
&\le 2\delta (E(w; g_{s\omega})-E(u(t); g_{s\omega})), 
\ \forall w\in H^1(M, X), \ 0<\delta<t.
\end{align}

Let $\eta\in C_0^2(\Omega\times (0,\infty))$ with $0\le\eta<\frac12$. 
Take $v=u_{sw}$, and as in the previous section define $u_{\eta(x,t)}(x,t)$ to be the unique point on the geodesic connecting $u(x,t)$ to $v(x,t)$ such that
\[
d(u_{\eta(x,t)}(x,t), u(x,t))=\eta(x,t)d(u(x,t), v(x,t)).
\]
By choosing $w(x,t)=u_{\eta(x,t)}(x,t)$ in \eqref{EVI1}
and $w(x,t)=v_{\eta(x,t)}(x,t)(=u_{1-\eta(x,t)}(x,t))$ in
\eqref{EVI2} and adding the resulting inequalities, we have
that
\begin{align}\label{EVI20}
&\int_M \Big(d^2(u(x, t), u_{\eta(x,t)}(x,t))
-d^2(u(x, t-\delta), u_{\eta(x,t)}(x,t))\Big)\,dv_g
\nonumber\\
&\ + \int_M \Big(d^2(v(x, t), v_{\eta(x,t)}(x,t))-d^2(v(x, t-\delta), v_{\eta(x,t)}(x,t))\Big)\,dv_{g_{s\omega}}\nonumber\\
&\le 2\delta\Big(E(u_{\eta(t)}(t); g)+E(v_{\eta(t}(t);g_{s\omega})
-E(u(t);g)-E(v(t); g_{s\omega})\Big).
\end{align}
By regrouping terms in \eqref{EVI20}, we would have
\begin{align}\label{EVI4}
&\int_M \Big(d^2(u(x, t), u_{\eta(x,t)}(x,t))
-d^2(u(x, t-\delta), u_{\eta(x,t)}(x,t))\Big)\,dv_g
\nonumber\\
&\ + \int_M \Big(d^2(v(x, t), v_{\eta(x,t)}(x,t))-d^2(v(x, t-\delta), v_{\eta(x,t)}(x,t))\Big)\,dv_g
\nonumber\\
&\le 2\delta\Big(\int_M \big(\pi_{u_{\eta(t)}(t)}+\pi_{v_{\eta(t}(t)}\big)_{ij}g^{ij}dv_g
-\int_M \big(\pi_{u(t)}+\pi_{v(t)}\big)_{ij} g^{ij}dv_g\Big)\nonumber\\
&\quad+2\delta\int_M \Big(\pi_{u_{1-\eta(x,t)}(x,t)}-\pi_{v(x,t)})_{ij}\Big)
\Big((g^{ij}dv_g)_{s\omega}-g^{ij}dv_g\Big)\nonumber\\
&\quad+\int_M \Big(d^2(v(x, t), v_{\eta(x,t)}(x,t))-d^2(v(x, t-\delta), v_{\eta(x,t)}(x,t))\Big)\Big(dv_g-dv_{g_{s\omega}}\Big)\nonumber\\
&=I+II+III.
\end{align}

Now we want to estimate $I, II, III$ separately as follows. 
We can deduce from \cite[Lemma 2.4.5, Lemma 2.4.2]{Korevaar-Schoen1993} that
\begin{align}\label{I-est}
I\le 2\delta\Big(-\int_M \nabla\eta\cdot\nabla d^2(u, v)\,dv_g
-\int_M \mathcal{C}(u,v,\eta)_{ij}g^{ij}\,dv_g +\int_M Q(\eta,\nabla \eta)\,dv_g\Big),    
\end{align}
where 
$$\mathcal{C}(u,v,\eta)=\pi_{u}+\pi_v-\mathcal{P}(u,v,\eta)-\mathcal{P}(u,v,1-\eta),$$ 
and $\mathcal{P}(u,v,\eta)$ is the symmetric, bilinear, integrable tensor given by \cite[Lemma 2.4.4]{Korevaar-Schoen1993}:
\[
\frac{d^2\big(u_{\eta(x,t)}(x,t), u_{\eta(x,t)}(\overline{x}(x,\eps),t)\big)}
{\eps^2}\,dv_g\rightharpoonup \mathcal{P}(u,v,\eta)(Z,Z)\,dv_g
\]
for any smooth tangent vector field $Z\in\Gamma(TM)$. Here
\[
\frac{d}{ds}\overline{x}(x,s)=Z(\overline{x}(x,s)), \ 
\overline{x}(x,0)=x.
\]

We can estimate $III$ by
\begin{align}\label{III-est}
|III|&\le \int_M \Big|d^2(v(x, t), v_{\eta(x,t)}(x,t))-d^2(v(x, t-\delta), v_{\eta(x,t)}(x,t))\Big|\Big|dv_g-dv_{g_{s\omega}}\Big|\\
&\le Cs\int_M d(v(x, t), v(x, t-\delta))
(d(v(x, t), v_{\eta(x,t)}(x,t))+d(v(x, t-\delta), v_{\eta(x,t)}(x,t)))\,dv_g.
\nonumber
\end{align}

The most difficult term is $II$, which has been 
estimated by \cite[Theorem 2.4.6]{Korevaar-Schoen1993}. Here we sketch it for convenience of the reader. First, decompose
\begin{align*}
   \pi_{u_{1-\eta}}-\pi_{v}
   =\big(\pi_{u_{1-\eta}}-\mathcal{P}(u,v,1-\eta)\big)
   +\big(\mathcal{P}(u,v,1-\eta)-\eta\pi_{u}-(1-\eta)\pi_v\Big) 
   -\eta\Big(\pi_v-\pi_u\Big).
\end{align*}
And we can write
\begin{align*}
\int_M \eta \big(\pi_v-\pi_u\big)_{ij}\Big((g^{ij}dv_g)_{s\omega}-g^{ij}dv_g\Big)
&=\int_M \eta \big(\pi_{u_{s\omega}}-\pi_u\big)_{ij}\Big((g^{ij}dv_g)_{s\omega}-g^{ij}dv_g\Big)\\
&=\int_M \eta(\pi_u)_{ij}\Big(2g^{ij}dv_g-(g^{ij}dv_g)_{s\omega}
-(g^{ij}dv_g)_{-s\omega}\Big)\\
&\quad+\int_M (\pi_u)_{ij}\big(\eta_{-s\omega}-\eta\big)
\big(g^{ij}dv_g-(g^{ij}dv_g)_{-s\omega}\big).
\end{align*}
Thus, we can obtain
\begin{align*}
\Big|\int_M \eta \big(\pi_v-\pi_u\big)_{ij}\Big((g^{ij}dv_g)_{s\omega}-g^{ij}dv_g\Big)\Big|
\le Cs^2\int_{M}(\eta+|\nabla\eta|)|\nabla u|^2\,dv_g,
\end{align*}
where $C>0$ depends on the $C^{1,1}$-norm of $g$ under the coordinate system, or equivalently the bound of curvature.
While
\begin{align*}
&\Big|\int_M \big(\pi_{u_{1-\eta}}-\mathcal{P}(u,v,1-\eta)\big) 
\Big((g^{ij}dv_g)_{s\omega}-g^{ij}dv_g\Big)\Big|\\
&\le Cs\int_M \Big(|\nabla\eta|d(u,v) (|\nabla u|_1+|\nabla v|_1)+|Q(\eta,\nabla\eta)|\Big)\,dv_g,
\end{align*}
and
\begin{align*}
\Big|\int_M\big(\mathcal{P}(u,v,1-\eta)-\eta\pi_{u}-(1-\eta)\pi_v\Big) 
\Big((g^{ij}dv_g)_{s\omega}-g^{ij}dv_g\Big)\Big|
\le Cs\int_M \mathcal{C}(u,v,\eta) \,dv_g. 
\end{align*}
Putting these estimate together, we obtain that
\begin{align}
|II|&\le C\delta\Big[s\int_M \Big(|\nabla\eta|d(u,v) (|\nabla u|_1+|\nabla v|_1)+|Q(\eta,\nabla\eta)|\Big)\,dv_g +s^2\int_{M}(\eta+|\nabla\eta|)|\nabla u|^2\,dv_g\nonumber\\
&\qquad\qquad\qquad+s\int_M |\mathcal{C}(u,v,\eta)| \,dv_g\Big].
\label{II-est}
\end{align}

Next we need to estimate the left hand side of the inequality \eqref{EVI4}.
For this, we claim that
\begin{align}\label{NPC00} 
&\int_{M} \eta(x,t)\Big(
d^2(u(x,t), v(x,t))-d^2(u(x,t-\delta), v(x,t-\delta))\Big)\,dv_g\\
&\le 
\int_M\Big[d^2(u(x, t), u_{\eta(x,t)}(x,t))-d^2(u(x, t-\delta), u_{\eta(x,t)}(x,t))
+d^2(v(x, t), v_{\eta(x,t)}(x,t))\nonumber\\
&\ \ \ \  -d^2(v(x, t-\delta), v_{\eta(x,t)}(x,t))
+ d^2(u(x, t), u(x, t-\delta))+d^2(v(x, t), v(x, t-\delta))\Big]\,dv_g.
\nonumber
\end{align}

To see \eqref{NPC00}, recall that, for $P, Q, S\in X$,
let $P_\lambda =(1-\lambda)P+\lambda Q$, for $0\le \lambda\le 1$, 
be the unique point on the geodesic connecting $P$ to $S$
such that
\[
d(P_\lambda, P)=\lambda d(P,S).
\]
the following inequality follows directly from \eqref{NPC1}
\begin{align}\label{NPC111}
\lambda\Big(d^2(P,Q)+d^2(P,S)-d^2(Q,S)\Big)
\le d^2(P, Q)+d^2(P, P_\lambda)-d^2(Q, P_\lambda).
\end{align}
Applying \eqref{NPC111} with $P=u(x,t), Q=u(x,t-\delta), S=v(x,t)$,
and $\lambda=\eta(x,t)$, we have 
\begin{align}\label{NPC01}
&\int_M \eta(x,t)\Big[
d^2(u(x,t), u(x,t-\delta))+d^2(u(x,t), v(x,t))-d^2(u(x,t-\delta), v(x,t)\Big]\,dv_g
\\
&\le \int_{M}
\Big[d^2(u(x, t), u(x, t-\delta))+d^2(u(x, t), u_{\eta(x,t)}(x,t))-d^2(u_{\eta(x, t)}(x,t), u(x, t-\delta))\Big]\,dv_g.\nonumber
\end{align}
Similarly, applying \eqref{NPC11} with $P=v(x,t), Q=v(x,t-s), S=u(x,t)$,
and $\lambda=\eta(x,t)$, we have 
\begin{align}\label{NPC02}
&\int_M \eta(x,t)\Big[
d^2(v(x,t), v(x,t-s))+d^2(v(x,t), u(x,t))-d^2(v(x,t-s), u(x,t)\Big]\,dv_g
\\
&\le \int_{M}
\Big[d^2(v(x, t), v(x, t-\delta))+d^2(v(x, t), v_{\eta(x,t)}(x,t))-d^2(v_{\eta(x, t)}(x,t), v(x, t-\delta))\Big]\,dv_g.\nonumber
\end{align}
Also, recall the following inequality (see \cite[(2.1vi)]{Korevaar-Schoen1993}):
\begin{equation}\label{NPC03}
d^2(P, R) + d^2(Q, S)-d^2(P, Q)-d^2(R, S)
\le d^2(Q,R)+d^2(P,S)-(d(R,S)-d(P,Q))^2,
\end{equation}
Applying \eqref{NPC03} with $P=u(x,t), Q=v(x,t), R=v(x,t-s)$
and $S=u(x,t-s)$, we can deduce that
\begin{align}\label{NPC04}
&d^2(u(x,t), u(x,t-\delta))+d^2(u(x,t), v(x,t))
+d^2(u(x,t), v(x,t-\delta))-d^2(u(x,t-\delta), v(x,t))\nonumber\\
&-d^2(u(x,t), v(x,t-\delta))
\ge -d^2(u(x,t-\delta), v(x,t-\delta)) .
\end{align}
Hence, with the help of \eqref{NPC04}, adding \eqref{NPC01}
and \eqref{NPC02} together yields \eqref{NPC00}.    

Finally, we can combine \eqref{EVI4}, \eqref{I-est}, \eqref{II-est},
\eqref{III-est}, and \eqref{NPC00} to obtain that
\begin{align}\label{EVI5}
&\int_{M} \eta(x,t)\Big(
d^2(u(x,t), v(x,t))-d^2(u(x,t-\delta), v(x,t-\delta))\Big)\,dv_g\\
&\le \int_M \big[d^2(u(x, t), u(x, t-\delta))+d^2(v(x, t), v(x, t-\delta))\big]\,dv_g\nonumber\\
&+2\delta\Big(-\int_M \nabla\eta\cdot\nabla d^2(u, v)\,dv_g
-\int_M \Big(\mathcal{C}(u,v,\eta)_{ij}g^{ij}-Cs |\mathcal{C}(u, v,\eta)|\Big)\,dv_g \nonumber\\
&+C\delta\int_M |Q(\eta,\nabla \eta)|\,dv_g+C\delta s^2\int_M (\eta+|\nabla\eta|)|\nabla u|^2\,dv_g\nonumber\\
&+C\delta s\int_M \Big(|\nabla\eta|d(u,v) (|\nabla u|_1+|\nabla v|_1)+|Q(\eta,\nabla\eta)|\Big)\,dv_g\nonumber\\
&+Cs\int_M d(v(x, t), v(x, t-\delta))
(d(v(x, t), v_{\eta(x,t)}(x,t))+d(v(x, t-\delta), v_{\eta(x,t)}(x,t)))\,dv_g.\nonumber
\end{align}
Observe that since $\mathcal{C}(u,v,\eta)$ is positive-definite, bilinear tensor, we have that for sufficiently small $s>0$, it holds 
\[
\int_M \Big(\mathcal{C}(u,v,\eta)_{ij}g^{ij}-Cs |\mathcal{C}(u, v,\eta)|\Big)\,dv_g\ge 0.
\]
Also it is easy to show  that
\begin{align*}
&\int_M \big[d^2(u(x, t), u(x, t-\delta))+d^2(v(x, t), v(x, t-\delta))\big]\,dv_g\\
&\le
C\delta\int_{t-\delta}^t \int_M \big(|\partial_t u|^2+|\partial_t v|^2\big)\,dv_gdt.
\end{align*}
Substituting these two inequalities into \eqref{EVI5} and dividing
both sides of the resulting inequality by $\delta$ and then sending
$\delta\to 0$, we conclude that
\begin{align}\label{EVI6}
&\int_{M} \eta(x,t)\partial_t d^2(u(x,t), v(x,t))\,dv_g\nonumber\\
&\le -2\int_M \nabla\eta\cdot\nabla d^2(u,v)\,dv_g 
+\int_M |Q(\eta,\nabla\eta)|\,dv_g\nonumber\\
&+C\Big[s\int_M \Big(|\nabla\eta|d(u,v) (|\nabla u|_1+|\nabla v|_1)+|Q(\eta,\nabla\eta)|\Big)\,dv_g
+s^2\int_{M}(\eta+|\nabla\eta|)|\nabla u|^2\,dv_g\Big]\nonumber\\
&+Cs\int_M \eta |\partial_t v(x,t)|_1 d(u(x,t), v(x,t))\,dv_g.
\end{align}
Now we can replace $\eta$ by $\eps\eta$ and send $\eps\to 0$ to obtain
that 
\begin{align}\label{EVI7}
&\int_{M} \eta(x,t)\partial_t d^2(u(x,t), u_{s\omega}(x,t))\,dv_g\nonumber\\
&\le -2\int_M \nabla\eta\cdot\nabla d^2(u,u_{s\omega})\,dv_g \nonumber\\
&+C\Big[s\int_M \Big(|\nabla\eta|d(u,v) (|\nabla u|_1+|\nabla u_{s\omega}|_1)\Big)\,dv_g
+s^2\int_{M}(\eta+|\nabla\eta|)|\nabla u|^2\,dv_g\Big]\nonumber\\
&+Cs\int_M \eta |\partial_t u_{s\omega}|_1(x,t) d(u(x,t), u_{s\omega}(x,t))\,dv_g.
\end{align}
Finally, dividing both sides by $s^2$ and sending $s\to 0$, we have that
\begin{align}\label{EVI7}
&\int_{M} \eta\partial_t |\nabla u|^2\,dv_g+2\int_M \nabla\eta\cdot\nabla |\nabla u|^2\,dv_g \nonumber\\
&\le C\int_{M}(\eta+|\nabla\eta|)|\nabla u|^2\,dv_g
+C\int_M \eta |\partial_t u|_1 |\nabla u|_1\,dv_g\nonumber\\
&\le C\int_M (\eta+|\nabla\eta|)|\nabla u|^2\,dv_g
+C\int_M \eta |\partial_t u|^2\,dv_g,
\end{align}
where we use the H\"older inequality in the last step.
Integrating over $t$, this is equivalent to

\begin{align}\label{EVI7}
&\iint_{M\times (0,\infty)} |\nabla u|^2(\partial_t \eta+2\Delta \eta+C\eta
+C|\nabla \eta|)\,dv_gdt\ge -C\iint_{M\times (0,\infty)}\eta |\partial_t u|^2\,dv_gdt,
\end{align}
holds for all nonnegative $\eta\in C^\infty_{0}(M\times (0,\infty))$.

We would like to point out that  
this parabolic inequality can yield the 
$L^\infty$ -estimate of $|\nabla u|^2$,
if we can obtain the $L^\infty$-estimate of $|\partial_t u|^2$,
which is shown in section 3 below.

\section{Sub-Caloricity of $|\partial_t u|^2$}
In this section, we sketch how to use the EVI \eqref{EVI1} 
to derive sub-caloricity of $|\partial_t u|^2$ for any suitable weak
solution to the heat flow of harmonic maps into CAT(0)-space
$(X,d)$. Note that this property was previously deduced by
\cite{ZhangZhu2026}. Our derivation is slightly different
from \cite{ZhangZhu2026}.

Observe that for any $\delta>0$, $v(x,t)=u(x,t+\delta):M\times [0,\infty)\to (X,d)$ is also a suitable weak solution of the heat flow of harmonic maps,
with the initial condition $v(x,0)=u(x,\delta)\in H^1(M, X)$.
For any $\phi\in C_0^\infty(M\times (0,\infty))$, with $0\le\phi\le 1$,
 let $u_{\phi(z)}(z)$ be the unique point on
the geodesic joining $u(z)$ to $v(z)$
for $z\in M\times (0,\infty)$, satisfying
\[
d(u_{\phi(z)}(z),u(z))=\phi(z) d(u(z), v(z)),
\]
and $v_{\phi(z)}(z)$ be the unique point on
the geodesic joining $u(z)$ to $v(z)$
for $z\in M\times (0,\infty)$, satisfying
\[
d(v_{\phi(z)}(z),v(z))=\phi(z) d(u(z), v(z)).
\]
It is easy to see that 
$$
v_{\phi(z)}(z)=u_{1-\phi(z)}(z), \  \forall \ z\in 
M\times (0,\infty).
$$

Applying \eqref{EVI2} to both solutions $u$ and $v$, we obtain
that for $0<s<t$, it holds
\begin{align}\label{EVI3}
&d_2^2(u(t), u_\phi(t))-d_2^2(u(t-s), u_\phi(t))
+d_2^2(v(t), v_\phi(t))-d_2^2(v(t-s), v_\phi(t))\nonumber\\
&\qquad\qquad\qquad\qquad\qquad\ \ \ \le 
2s\big[E(u_\phi(t))-E(u(t))+E(v_\phi(t))-E(v(t))\big].    
\end{align}

Now we claim that
\begin{align}\label{energy_comparison}
 E(u_\phi(t))-E(u(t))+E(v_\phi(t))-E(v(t)) 
\le -\int_{M\times\{t\}}\nabla\phi\cdot\nabla\big((1-2\phi)d^2(u, v)\big)\,dv_g
\end{align}
holds for any $\phi\in C_0^\infty(M\times (0,\infty))$ with
$0\le\phi\le 1$.

The proof of \eqref{energy_comparison} is based on Reshetnyak's quadrilateral comparison property for $d^2$, see \cite[Theorem 2.1.2]{Korevaar-Schoen1993}.  The original proof was given 
by \cite[Lemma 3.2]{ZhangZhu2026}. Here we sketch a slightly different derivation just for convenience of the interested reader. 

For $x,y\in M, t>0$, denote $P=u(x,t), Q=u(y,t),
R=v(y,t), S=v(x,t)$, $\lambda=\phi(x,t)$, $\mu=\phi(y,t)$,
$P_\lambda=u_{\phi}(x,t)$, $P_{1-\lambda}=v_\phi(x,t)$,
$Q_\mu=u_{\phi}(y,t)$, and $Q_{1-\mu}=v_\phi(y,t)$.
Then 
\begin{align}\label{comparison10}
&d^2(P_\lambda, Q_\mu)+d^2(P_{1-\lambda}, Q_{1-\mu})\nonumber\\
&\le \Big[\mu(1-\lambda)d^2(P,R)+\lambda(1-\mu)d^2(Q,S)+\lambda\mu d^2(R,S)
\nonumber\\
&\ \ \ +(1-\lambda)(1-\mu)d^2(P,Q)-\lambda(1-\lambda)d^2(P,S)-\mu(1-\mu)d^2(Q,R)\Big]\nonumber\\
&\ \ \ +\Big[\lambda(1-\mu)d^2(P,R)+\mu(1-\lambda)d^2(Q,S)+(1-\lambda)(1-\mu) d^2(R,S)
\nonumber\\
&\ \ \ + \ \lambda\mu d^2(P,Q)-\lambda(1-\lambda)d^2(P,S)-\mu(1-\mu)d^2(Q,R)\Big]\nonumber\\
&=d^2(P,Q)+d^2(R,S)+(\lambda+\mu-2\lambda\mu)\Big[d^2(P,R)+d^2(Q,S)-d^2(P,Q)-d^2(R,S)\Big]\nonumber\\
&\ \ \ -2\lambda(1-\lambda) d^2(P, S)-2\mu(1-\mu)d^2(Q,R)\nonumber\\
&\le d^2(P,Q)+d^2(R,S)+(\lambda+\mu-2\lambda\mu)\Big[d^2(Q,R)+d^2(P,S)
-(d(R,S)-d(P,Q))^2\Big]\nonumber\\
&\ \ \ -2\lambda(1-\lambda) d^2(P, S)-2\mu(1-\mu)d^2(Q,R)\nonumber\\
&\le d^2(P,Q)+d^2(R,S)-(\lambda-\mu)\Big[(1-2\lambda)d^2(P,S)-(1-2\mu)d^2(Q,R)\Big]
\nonumber\\
&\ \ \ -[\lambda+\mu-2\lambda\mu](d(R,S)-d(P,Q))^2\nonumber\\
&\le d^2(P,Q)+d^2(R,S)-(\lambda-\mu)\Big[(1-2\lambda)d^2(P,S)-(1-2\mu)d^2(Q,R)\Big],
\end{align}
where have used the following inequality (see \cite[(2.1vi)]{Korevaar-Schoen1993}):
\begin{equation}\label{NPC9}
d^2(P, R) + d^2(Q, S)-d^2(P, Q)-d^2(R, S)
\le d^2(Q,R)+d^2(P,S)-(d(R,S)-d(P,Q))^2,
\end{equation}
and 
\[
\lambda+\mu-2\lambda\mu\ge \lambda^2+\mu^2-2\lambda\mu=(\lambda-\mu)^2\ge 0,
\]
since $0\le\lambda, \mu\le 1$.

It is readily seen, after substituting the values 
of $P, Q, R, S$ and
$\lambda, \mu$, that \eqref{comparison10} yields
\begin{align}\label{comparison11}
&d^2(u_\phi(x,t), u_\phi(y,t))
+d^2(v_\phi(x,t), v_\phi(y,t))
-d^2(u(x,t), u(y,t))-d^2(v(x,t), v(y,t))\\
&\le
-\big(\phi(x,t)-\phi(y,t)\big)\big[
(1-2\phi(x,t))d^2(u(x,t),v(x,t))-(1-2\phi(y,t))d^2(u(y,t),v(y,t))\big].\nonumber
\end{align}
This, combined with the definition of energy density
of the maps $u, v, u_\phi, v_\phi$ by \cite{Korevaar-Schoen1993}, immediately
implies \eqref{energy_comparison}.
    
Next we claim that for any $0<s<t$ and $\phi\in C_0^\infty(M\times (0,\infty))$ with $0\le\phi\le 1$, it holds 
\begin{align}\label{subcaloricity1}
&\int_{M} \phi(x,t)\Big(
d^2(u(x,t), v(x,t))-d^2(u(x,t-s), v(x,t-s))\Big)\,dv_g\nonumber\\
&\le d^2_2(u(t), u(t-s))+d^2_2(v(t), v(t-s))\nonumber\\
&\ \ \ +2s\Big(E(u_\phi(t))+E(v_\phi(t))-E(u(t))-E(v(t))\Big).
\end{align}
\eqref{subcaloricity1} follows from \eqref{EVI3}, if we can show
\begin{align}\label{NPC10}    
&\int_{M} \phi(x,t)\Big(
d^2(u(x,t), v(x,t))-d^2(u(x,t-s), v(x,t-s))\Big)\,dv_g\nonumber\\
&\le 
d_2^2(u(t), u_\phi(t))-d_2^2(u(t-s), u_\phi(t))
+d_2^2(v(t), v_\phi(t))-d_2^2(v(t-s), v_\phi(t))\nonumber\\
&\ \ \ + d^2_2(u(t), u(t-s))+d^2_2(v(t), v(t-s)).
\end{align}
To see \eqref{NPC10}, first recall the inequality \eqref{NPC111}:
\[
\lambda\Big(d^2(P,Q)+d^2(P,S)-d^2(Q,S)\Big))
\le d^2(P, Q)+d^2(P, P_\lambda)-d^2(Q, P_\lambda).
\]
Choosing $P=u(x,t), Q=u(x,t-s), S=v(x,t)$, and $\lambda=\phi(x,t)$,
we obtain that
\begin{align}\label{NPC11}
&\int_M \phi(x,t)\Big[
d^2(u(x,t), u(x,t-s))+d^2(u(x,t), v(x,t))-d^2(u(x,t-s), v(x,t)\Big]\,dv_g
\nonumber\\
&\qquad\qquad\qquad\le d^2_2(u(t), u(t-s))+d^2_2(u(t), u_\phi(t))-d^2_2(u_\phi(t), u(t-s)),
\end{align}
and, similarly, by choosing $P=v(x,t), Q=v(x,t-s), S=u(x,t)$,
and $\lambda=\phi(x,t)$, we have that
\begin{align}\label{NPC12}
&\int_M \phi(x,t)\Big[
d^2(v(x,t), v(x,t-s))+d^2(v(x,t), u(x,t))-d^2(v(x,t-s), u(x,t)\Big]\,dv_g
\nonumber\\
&\qquad\qquad\qquad\le d^2_2(v(t), v(t-s))+d^2_2(v(t), v_\phi(t))-d^2_2(v_\phi(t), v(t-s)).
\end{align}
From \eqref{NPC9}, we have 
\begin{align}\label{NPC13}
&d^2(u(x,t), u(x,t-s))+d^2(u(x,t), v(x,t))
+d^2(u(x,t), v(x,t-s))-d^2(u(x,t-s), v(x,t))\nonumber\\
&-d^2(u(x,t), v(x,t-s))
\ge -d^2(u(x,t-s), v(x,t-s)) .
\end{align}
Hence, with the help of \eqref{NPC13}, adding \eqref{NPC11}
and \eqref{NPC12} together yields \eqref{NPC10}.

Observe that
\[
d^2_2(u(t), u(t-s))\le 
Cs\int_{t-s}^t \int_M |\partial_t u(x,\tau)|^2\,dv_gd\tau,
\]
\[
d^2_2(v(t), v(t-s))\le Cs\int_{t-s}^t 
\int_M |\partial_t u(x,\tau)|^2\,dv_gd\tau.
\]
Hence, after dividing both sides of \eqref{subcaloricity1} by $s$
and sending
$s\to 0^+$ and applying \eqref{energy_comparison},  we obtain 
\begin{align}\label{subcaloricity2}
&\int_M \partial_t d^2(u(x,t), v(x,t))\phi(x,t)\,dv_g\nonumber\\
&\le -2\int_{M}\nabla\phi(x,t)\cdot\nabla\Big((1-2\phi(x,t))d^2(u(x,t), v(x,t))\Big)\,dv_g.
\end{align}
If we first substitute $\phi$ by $\varepsilon\phi$ in \eqref{subcaloricity2} and then divide both sides of the resulting
inequality by $\eps$, and then send $\varepsilon\to 0$, and finally integrate the resulting inequality
over $t\in (0,\infty)$, we arrive at
\begin{align}\label{subcaloricity3}
&\iint_{M\times (0,\infty)} \partial_t d^2(u(x,t), v(x,t))\phi(x,t)\,dv_gdt
\nonumber\\
&\le -2\iint_{M\times (0,\infty)}\nabla\phi(x,t)\cdot\nabla \big(d^2(u(x,t), v(x,t))\big)\,dv_gdt.
\end{align}
Since $v(x,t)=u(x,t+\delta)$, \eqref{subcaloricity3} is equivalent to
\begin{align}\label{subcaloricity4}
(\partial_t-2\Delta)   d^2(u(x,t), u(x,t+\delta))\ge 0  
\end{align}
holds weakly in $M\times (0,\infty)$. 

After dividing both sides of \eqref{subcaloricity4} by $\delta^2$ and 
then sending $\delta\searrow 0$, we obtain that $|\partial_t u|^2$
weakly solves
\begin{align}\label{subcaloricity5}
(\partial_t-2\Delta)|\partial_t u|^2\ge 0
\end{align}
in $M\times (0,\infty)$.

\noindent{\bf Proof of Theorem \ref{Lip1}}: Since
$0\le f_\delta(x,t)=d^2(u(x,t), u(x,t+\delta))\in H^1_{\rm{loc}}(M\times (0,\infty))$ weakly solves \eqref{subcaloricity4}, it follows from the
Harnack inequality for the heat equation that
for any $z_0=(x_0,t_0)\in M\times (0,\infty)$, 
\begin{equation}\label{Harnack1}
\sup_{P_R(z_0)} d^2(u(x,t), u(x,t+\delta))
\le CR^{-(n+2)}\int_{P_{2R}(z_0)}d^2(u(y,\tau), u(y,\tau+\delta))
\,dv_g d\tau
\end{equation}
holds for all $0<R<\frac12 \min\big\{\sqrt{t_0}, \ {\rm{inj}}(M)\big\}.$
Here $P_R(z_0)=B_R(x_0)\times (t_0-R^2, t_0+R^2)$ is the parabolic ball
centered at $z_0$ with radius $R$. 

Dividing both sides of \eqref{Harnack1} by $\delta^2$ and sending
$\delta>0$, we obtain the local $L^\infty$-bound of $|\partial_t u|^2$:
\begin{equation}\label{Harnack2}
\sup_{P_R(z_0)} |\partial_t u|^2
\le CR^{-(n+2)}\int_{P_{2R}(z_0)}|\partial_t u|^2\,dv_g dt
\le CR^{-(n+2)} E(u_0).
\end{equation}

Substituting the estimate \eqref{Harnack1} into \eqref{EVI7}, we obtain
that 
\begin{align}\label{EVI8}
&\iint_{P_R(z_0)} |\nabla u|^2\Big[\big(\partial_t+2\Delta\big) \eta+C\eta
+C|\nabla \eta|\Big]\,dv_gdt\ge -C_R\iint_{P_R(z_0)}\eta\,dv_gdt
\end{align}
holds for all nonnegative $\eta\in C^\infty_{0}(P_R(z_0))$, with $C_R
=CR^{-(n+2)}E(u_0)>0$. 

Now we can follow the proof of Moser
\cite{Moser1964} to show that $|\nabla u|^2$ is bounded
in $P_{\frac{R}2}(z_0)$, and
\begin{align}\label{Harnack3}
\sup_{P_{\frac{R}2}(z_0)} |\nabla u|^2
&\le C\Big(R^{-(n+2)}\int_{P_{R}(z_0)}|\nabla u|^2\,dv_g dt
+R^{-(n+2)}E(u_0)\Big)\nonumber\\
&\le CR^{-(n+2)}(1+R^2) E(u_0)
\end{align}
holds for all $0<R<\frac12 \min\big\{\sqrt{t_0},\ {\rm{inj}}(M)\big\}.$
This completes the proof. \qed

\bigskip
\section*{Acknowledgements}
The first author is partially supported by NSF DMS 2247773.
The second author is partially supported by NSF DMS 2453789 and Simons Travel Grant TSM-00007723. 
The authors would like to thank Xi-Ping Zhu and Hui-Chun Zhang for your interests and comments on the paper.

\end{document}